\documentclass[preprint,12pt]{elsarticle}

\usepackage{amsthm}
\usepackage{amsmath,amssymb,txfonts}
\usepackage{color,bm,comment}

\newtheorem{theorem}{Theorem}
\newtheorem{lemma}{Lemma}
\newtheorem{prop}{Proposition}

\newtheorem{remark}{Remark}

\theoremstyle{definition}

\def\re{\mathbb{R}}
\def\N{\mathbb{N}}

\def\({\left(}
\def\){\right)}
\def\[{\left[}
\def\]{\right]}
\def\pd{\partial}
\def\lap{\Delta}

\def\ep{\varepsilon}
\def\w{\omega}
\def\la{\lambda}

\def\D{\mathcal{D}}
\def\M{\mathcal{M}}

\begin{document}

\begin{frontmatter}



\title{Minimization problem associated with an improved Hardy-Sobolev type inequality}


\author[M]{Megumi Sano}
\ead{smegumi@hiroshima-u.ac.jp}
\address[M]{Laboratory of Mathematics, Graduate School of Engineering,
Hiroshima University, Higashi-Hiroshima, 739-8527, Japan}


\date{\today}

\begin{abstract}
We consider the existence and the non-existence of a minimizer of the following minimization problems associated with an improved Hardy-Sobolev type inequality introduced by Ioku \cite{I}.
\begin{align*}
I_a := \inf_{u \in W_0^{1,p}(B_R ) \setminus \{ 0\} } \dfrac{\int_{B_R} |\nabla u |^{p} \,dx}{\( \int_{B_R} |u|^{p^*(s)} V_a(x) \,dx \)^{\frac{p}{p^*(s)}}}, \,\,\text{where}\,\, V_a (x) =\frac{1}{|x|^s \( 1- a \,\( \frac{|x|}{R} \)^{\frac{N-p}{p-1}} \)^\beta} \ge \frac{1}{|x|^s}.
\end{align*}
Only for radial functions, the minimization problem $I_a$ is equivalent to it associated with the classical Hardy-Sobolev inequality on $\re^N$ via a transformation. 
First, we summarize various transformations including that transformation and give a viewpoint of such transformations. 
As an application of this viewpoint, we derive {\it an infinite dimensional form} of the classical Sobolev inequality in some sense. 
Next, without the transformation, we investigate the minimization problems $I_a$ on balls $B_R$. 
In contrast to the classical results for $a=0$, we show the existence of non-radial minimizers for the Hardy-Sobolev critical exponent $p^* (s)=\frac{p (N-s)}{N-p}$ on bounded domains. 
Finally, we give remarks of a different structure between two nonlinear scalings which are equivalent to the usual scaling only for radial functions under some transformations. 
\end{abstract}

\end{frontmatter}



%
%
\section{Introduction and main results}\label{intro}

Let $B_R \subset \re^N, 1 < p < N, 0 \le s \le p, p^* (s) = \frac{p (N-s)}{N-p}$ and $p^* = p^*(0)$. 
Then the classical Hardy-Sobolev inequality:
\begin{equation}
\label{HS}
C_{N, p, s} \( \int_{B_R} \frac{|u|^{p^* (s)}}{|x|^s} dx \)^{\frac{p}{p^* (s)}} \le \int_{B_R} | \nabla u |^p dx
\end{equation}
holds for all $u \in W^{1,p}_0(B_R)$, where $W_0^{1,p}(B_R)$ is the completion of $C_c^{\infty}(B_R)$ with respect to the norm $\| \nabla (\cdot )\|_{L^p(B_R)}$ and $C_{N,p,s}$ is the best constant of (\ref{HS}). 
In the case where $s=0$ (resp. $s=p$), the inequality (\ref{HS}) is called the Sobolev (resp. Hardy) inequality. The Hardy-Sobolev inequality (\ref{HS}) is quite fundamental and important since it expresses the embeddings of the Sobolev spaces $W_0^{1,p}$. Furthermore the variational problems and partial differential equations associated with the Hardy-Sobolev inequality (\ref{HS}) are well-studied by many mathematicians so far, see \cite{Au}, \cite{T}, \cite{L}, \cite{BG}, \cite{BV}, \cite{VZ}, to name a few.

Recently, Ioku \cite{I} showed the following improved Hardy-Sobolev inequality for radial functions via the transformation (\ref{trans}), see \S \ref{Trans}. 
\begin{equation}\label{IHS}
C_{N, p, s} \( \int_{B_R} \frac{|u|^{p^* (s)}}{|x|^s \( 1- \( \frac{|x|}{R} \)^{\frac{N-p}{p-1}} \)^{\beta}} dx \)^{\frac{p}{p^* (s)}} \le \int_{B_R} | \nabla u |^p dx\,\,\text{for} \,\, u \in W_{0, {\rm rad}}^{1,p}(B_R),
\end{equation}
where $\beta = \beta (s)= \frac{(N-1)p - (p-1)s}{N-p}$. 
One virtue of (\ref{IHS}) is that we can take a limit directly for the improved inequality (\ref{IHS}) as $p \nearrow N$, differently from the classical one. Indeed, Ioku \cite{I} showed that the limit of the improved inequality (\ref{IHS}) with $s=0$ is the Alvino inequality \cite{Al} which implies the optimal embedding of $W_0^{1,N}(B_R)$ into the Orlicz spaces, and also the limit of the improved inequality (\ref{IHS}) with $s=p$ is the critical Hardy inequality which implies the embedding of $W_0^{1,N} (B_R)$ into the Lorentz-Zygmund spaces $L^{\infty, N}(\log L)^{-1}$ which is smaller than the Orlicz space. 
For some indirect limiting procedures for the classical Hardy-Sobolev inequalities, see \cite{Tru}, \cite{BP}, \cite{SS} and a survey \cite{S(RIMS)}. 
Based on the transformation (\ref{trans}), the improved inequality (\ref{IHS}) on $B_R$ equivalently connects to the classical one  (\ref{HS}) on the whole space $\re^N$. This yields that the improved inequality (\ref{IHS}) has the scale invariance under a scaling to which the usual scaling is changed via the transformation (\ref{trans}), and there exists a radial minimizer of (\ref{IHS}) when $0 \le s < p$. For more details, see \cite{I} or \S \ref{Trans}.


In this paper, without the transformation, we investigate the following extended minimization problems $I_a$ for $a \in [0,1]$ associated with improved Hardy-Sobolev inequalities. 
\begin{align*}
I_{a} := \inf_{u \in W_{0}^{1,p}(B_R) \setminus \{ 0\} } \frac{\int_{B_R} | \nabla u |^p \,dx}{\( \int_{B_R} |u|^{p^* (s)} V_a (x) \,dx \)^{\frac{p}{p^* (s)}}}, \,\text{where}\, V_a(x) =\frac{1}{|x|^s \( 1- a \,\( \frac{|x|}{R} \)^{\frac{N-p}{p-1}} \)^\beta} \ge \frac{1}{|x|^s}.
\end{align*}
Note that the potential function $V_a(x)$ also has the boundary singularity when $a=1$. Due to the boundary singularity, $I_1 = 0$ if $a=1$ and $s < p$, see Proposition \ref{Posi I_a} in \S \ref{Proof}. Therefore we exclude the case where $a=1$ and $s < p$. 

Our main results are as follows.


\begin{theorem}\label{Main} (i) Let $s=p$. Then for any $a \in [0, 1]$, $I_a =I_{a, {\rm rad}} = C_{N, p, p} = (\frac{N-p}{p})^p$ and $I_a$ is not attained.\\
(ii) Let $s=0$. Then $I_a =I_{a, {\rm rad}} (1-a)^{\frac{N-1}{N} p}= C_{N, p, 0} (1-a)^{\frac{N-1}{N} p}$ and $I_a$ is not attained for any $a \in [0,1)$.\\
(iii) Let $0 < s < p$. Then there exists $a_* \in (A, 1)$ such that $I_a < I_{a, {\rm rad}}$ for $a \in (a_*, 1)$, $I_a$ is attained for $a \in (a_*, 1)$, $I_a=I_{a, {\rm rad}}$ for $a \in [0, a_*]$ and $I_a$ is not attained for $a \in [0, a_*)$, where $A=1- \( \frac{s (p-1)}{p (N-1)} \)^{\frac{s}{\beta}} \left[ 1- \( \frac{s (p-1)}{p (N-1)} \)\right]$.
\end{theorem}

\begin{remark}
Note that $\frac{s (p-1)}{p(N-1)} < A$ and the potential function $V_a(x)$ is not monotone-decreasing with respect to $|x|$ for $a \in (\frac{s(p-1)}{p(N-1)}, 1]$. Therefore it seems difficult to reduce the radial setting due to the lack of the rearrangement technique, see the first part in \S \ref{Proof}. However we can show $I_a=I_{a, {\rm rad}}$ even for $a \in (\frac{s (p-1)}{p(N-1)}, a_*]$ thanks to the special shape of the potential function $V_a(x)$, see the last part of the proof of Theorem \ref{Main} (iii).  
\end{remark}


Our minimization problem $I_a$ is related to the following nonlinear elliptic equation with the singular potential $V_a(x) \ge |x|^{-s}$.
\begin{align}\label{EL}
\begin{cases}
-\text{div} \,( \, |\nabla u|^{p-2} \nabla u \,) = b V_a(x) |u|^{p^*(s) -2}u \quad &\text{in} \,\, B_R, \\
\qquad u = 0  &\text{on} \,\, \pd B_R.
\end{cases}
\end{align}
The minimizer for $I_a$ is a ground state solution of the Euler-Lagrange equation (\ref{EL}) with a Lagrange multiplier $b$. 
Since the minimizer of Theorem \ref{Main} (iii) is a non-radial function, we can observe that the symmetry breaking phenomenon of the ground states of the elliptic equation (\ref{Eq}) occurs when $f(x)= V_a(x)$, $a$ is close to $1$, $s \in (0,p)$ and $p<q=p^*(s) < p^*$.
\begin{align}\label{Eq}
\begin{cases}
-\text{div} \,( \, |\nabla u|^{p-2} \nabla u \,) = f(x) \,|u|^{q -2}u \quad &\text{in} \,\, B_R, \\
\hspace{6.5em} u = 0  &\text{on} \,\, \pd B_R.
\end{cases}
\end{align}
It is well-known that the symmetry breaking phenomenon of the ground states of (\ref{Eq}) occurs when $f(x)=\( \frac{|x|}{R} \)^\alpha$, $\alpha >0$ is sufficiently large and $p<q < p^*$ (ref. \cite{SSW}). Theorem \ref{Main} (iii) gives another example of the potential function $f(x)$,  which is not monotone-increasing and is not bounded differently from the H\'enon potential function $\( \frac{|x|}{R} \)^\alpha$.


In the case where $a=0$, Theorem \ref{Main} is the same as the classical result which is the non-existence of the minimizer of the classical Hardy-Sobolev inequality on bounded domains. However, in Theorem \ref{Main} (iii), we see that there exists a minimizer of $I_a$ even on bounded domains. This is different from the classical result. Concretely, we show that
$I_{a, \text{rad}}$ is the concentration level of minimizing sequence of $I_a$ for $s \in (0,p)$, where $I_{a, \text{rad}}$ is a level of $I_a$ only for radial functions. 
By concentration-compactness alternative, we can obtain a minimizer of $I_a$ for $a \in (a_*, 1)$. 
Note that the continuous embedding: $W_0^{1,p}(B_R) \hookrightarrow L^{p^*(s)}(B_R; V_a(x)\,dx)$ related to our problem is not compact due to the existence of a non-compact sequence given by the nonlinear scaling (\ref{new scale}) in \S \ref{Trans}, see also \S \ref{App}. 
However, the parameter of $a$ plays a role of lowering the level of $I_a$ than $I_{a, \text{rad}}$, and thanks to it, we can remove the possibility of occurring such non-compact behavior. 


This paper is organized as follows:
In \S \ref{Trans}, we summarize various transformations including the transformation and give a viewpoint of such transformations by which the classical Hardy-Sobolev inequality equivalently connects to another inequality.  As an application of this viewpoint, we derive {\it an infinite dimensional form} of the classical Sobolev inequality in some sense. 
In \S \ref{Proof}, we prepare several Lemmas and Propositions and show Theorem \ref{Main}. 
In \S \ref{App}, we give remarks of a different structure between two nonlinear scalings which are equivalent to the usual scaling only for radial functions under some transformations. 


We fix several notations: 
$B_R$ or $B_R^N$ denotes a $N$-dimensional ball centered $0$ with radius $R$. As a matter of convenience, we set $B_\infty^N = \re^N$ and $\frac{1}{\infty} = 0$.  
$\omega_{N-1}$ denotes an area of the unit sphere $\mathbb{S}^{N-1}$ in $\re^N$. $|A|$ denotes the Lebesgue measure of a set $A \subset \re^N$ and 
$X_{{\rm rad}} = \{ \, u \in X \, | \, u \,\,\text{is radial} \, \}$. 
The Schwarz symmetrization $u^{\#} \colon \re^N \to [0, \infty]$ of $u$ is given by 
\begin{align*}
u^{\#}(x) = u^{\#}(|x|)=\inf \left\{ \tau >0 \,: \, | \{ y \in \re^N \, :\, |u(y)| > \tau \} \,| \le |B_{|x|}(0) | \right\}.
\end{align*}
Throughout the paper, if a radial function $u$ is written as $u(x) = \tilde{u}(|x|)$ by some function $\tilde{u} = \tilde{u}(r)$, we write $u(x)= u(|x|)$ with admitting some ambiguity.

%
%

\section{Various transformations and an infinite dimensional form of the classical Sobolev inequality}\label{Trans}

First, we explain the transformation introduced by Ioku \cite{I} for readers convenience and give several remarks. 

We use the polar coordinates: $\re^N \ni x = r \w \,(r \in \re_+, \, \w \in \mathbb{S}^{N-1})$. 
Let $T \in [R, \infty], y \in B_T^N, x \in B_R^N, r= |x|, t =|y|$, and $w \in C_c^1 (B_T)$.
By using the fundamental solution of $p-$Laplacian with Dirichlet boundary condition, 
we consider the following transformation
\begin{align}\label{trans}
u(r \w) = w(t \w), \,\,\text{where}\,\, r^{-\frac{N-p}{p-1} } - R^{-\frac{N-p}{p-1} } = t^{-\frac{N-p}{p-1} } - T^{-\frac{N-p}{p-1} }.
\end{align}
Note that in the case where $T = \infty$ the transformation (\ref{trans}) is founded by Ioku\cite{I}. Then we see that
\begin{align*}
\int_{B_T} | \nabla w |^p \,dy 
&= \int_{\mathbb{S}^{N-1}} \int_0^T \left| \frac{\pd w}{\pd t} \w + \frac{1}{t} \nabla_{\mathbb{S}^{N-1}} w \,\right|^p t^{N-1} \,dt dS_{\w}\\
&= \int_{\mathbb{S}^{N-1}} \int_0^R \left| \frac{\pd u}{\pd r} \w + \frac{1}{t} \frac{dt}{dr} \nabla_{\mathbb{S}^{N-1}} u \,\right|^p \( \frac{dr}{dt} \)^{p-1} t^{N-1} \,dr dS_{\w}.
\end{align*}
Since 
\begin{align*}
\frac{dr}{dt} = \(\, \frac{r}{t} \,\)^{\frac{N-1}{p-1}}, \,\, t \, \frac{dr}{dt} = r^{\frac{N-1}{p-1}} \( r^{-\frac{N-p}{p-1} } - R^{-\frac{N-p}{p-1} } + T^{-\frac{N-p}{p-1} } \),
\end{align*}
we have
\begin{align}\label{nab non-rad}
\int_{B_T} | \nabla w |^p \,dy = \int_{B_R} | L_p u |^p \,dx,
\end{align}
where
\begin{align*}
L_p u = \frac{\pd u}{\pd r} \w + \frac{1}{r} \nabla_{\mathbb{S}^{N-1}} u \left[ 1 -  a \( \frac{r}{R} \)^{\frac{N-p}{p-1} } \right]^{-1}.
\end{align*}
where $a= 1- \( \frac{R}{T} \)^{\frac{N-p}{p-1}} \nearrow 1$ as $T \nearrow \infty$. 
Note that the differential operator $L_p$ is not $\nabla$ for $T > R$ due to the last term. However, if $u$ and $w$ are radial functions, then we can obtain the equality of two $L^p$ norms of $\nabla$ between $B_T$ and $B_R$ as follows:
\begin{align}\label{nab}
\int_{B_T} | \nabla w |^p \,dy = \int_{B_R} | \nabla u |^p \,dx \quad \text{if} \,\, u\,\,\text{and}\,\,w \,\, \text{are radial.}
\end{align}
On the other hand, for the Hardy-Sobolev term, we have
\begin{align}
\label{deno}
\int_{B_T} \frac{| w |^{p^*(s)}}{|y|^s} \,dy &= \int_{B_R} \frac{| u |^{p^*(s)}}{|x|^s\( 1-a \(\frac{|x|}{R}\)^{\frac{N-p}{p-1}}\)^{\beta}} \,dx,
\end{align}
where $\beta = \beta (s)= \frac{(N-1)p - (p-1)s}{N-p}$.
Set 
\begin{align*}
I_{a,{\rm rad}} := \inf_{u \in W_{0, {\rm rad}}^{1,p}(B_R) \setminus \{ 0\} } \frac{\int_{B_R} | \nabla u |^p \,dx}{\( \int_{B_R} \frac{| u |^{p^*(s)}}{|x|^s \( 1-a \,\( \frac{|x|}{R} \) ^{\frac{N-p}{p-1}} \)^{\beta}} \,dx \)^{\frac{p}{p^*(s)}}}.
\end{align*}
From (\ref{nab}) and (\ref{deno}), we observe that the minimization problem $I_{a, {\rm rad}}$ can be reduced the classical Hardy-Sobolev minimization problem $C_{N, p, s}$:
\begin{align*}
C_{N, p, s} 
= \inf_{w \in W_{0, {\rm rad}}^{1,p}(B_T ) \setminus \{ 0\} } \dfrac{\int_{B_T} | \nabla w |^p \,dy}{\( \int_{B_T} \frac{| w |^{p^*(s)}}{|y|^s} \,dy \)^{\frac{p}{p^*(s)}}}.
\end{align*}
It is well-known that 
$C_{N, p, s}$ is independent of the radius $T$ and $C_{N, p, s}$ is attained if and only if $T = \infty$ and $0\le s < p$. Moreover its minimizer is the family of
\begin{align*}
W_\la (t)= \la^{\frac{N-p}{p}} (1+(\la t)^{\frac{p-s}{p-1}})^{-\frac{N-p}{p-s}}\,\, \text{for}\,\,\la \in (0, \infty),
\end{align*}
see e.g. \cite{CRT}. Therefore we can obtain the following results for $I_{a, {\rm rad}}$ based on the transformation (\ref{trans}).

\begin{prop}\label{I_a rad}(\cite{I})
$I_{a, {\rm rad}}$ is independent of $a \in [0,1]$ and $I_{a, {\rm rad}} = C_{N, p, s}$.  And $I_{a, {\rm rad}}$ is attained if and only if $a=1$ and $0\le s < p$. Moreover, the minimizer of $I_{1, {\rm rad}}$ is the family of
\begin{align*}
U^{\la} (r) = \la^{\frac{N-p}{p}} \left[  1+ (\la r)^{\frac{p-s}{p-1}} \left\{  1- \( \frac{r}{R}\)^{\frac{N-p}{p-1}} \right\}^{-\frac{p-s}{N-p}} \right]^{-\frac{N-p}{p-s}}\,\, \text{for}\,\,\la \in (0, \infty).
\end{align*}
\end{prop}


\begin{remark}[Scale invariance]
It is well-known that thanks to the zero extension, the classical inequality (\ref{HS}) has the scale invariance under the usual scaling for $\la \in [1, \infty)$.
\begin{align}\label{usual scale}
w_\la (t \w) = 
\begin{cases}
\la^{\frac{N-p}{p}} w(\la t \w) \,\, &\text{for}\,\, t \in [0, \frac{T}{\la}], \\
0 &\text{for} \,\, t \in (\frac{T}{\la}, T].
\end{cases}
\end{align}
Note that in the case where $T = \infty$, we can consider any $\la \in (0, \infty)$. 
By the transformation (\ref{trans}), the usual scaling (\ref{usual scale}) is changed its form to the following scaling for $\la \in [1, \infty)$.
\begin{align}\label{new scale}
&u^\la (r \w) = 
\begin{cases}
\la^{\frac{N-p}{p}} u(\tilde{r} \w) \,\, &\text{for}\,\, t \in [0, \tilde{R}], \\
0 &\text{for} \,\, t \in (\tilde{R}, R],
\end{cases} \\
&\text{where}\,\,\tilde{r} =  (\la r) \left[ 1 + a \( \la^{\frac{N-p}{p-1}} -1 \)  \( \frac{r}{R} \)^{\frac{N-p}{p-1}}  \right]^{\,-\frac{p-1}{N-p}},\, \tilde{R} = R \( \la^{\frac{N-p}{p-1}} (1-a) + a \)^{-\frac{p-1}{N-p}}. \notag
\end{align}
Respectively, in the case where $a=1$, we can consider any $\la \in (0, \infty)$.   
Obviously from (\ref{nab non-rad}) and (\ref{deno}), we see that the improved Hardy-Sobolev inequality:
\begin{align*}
C_{N, p, s} \( \int_{B_R} \frac{| u |^{p^*(s)}}{|x|^s\( 1-a \(\frac{|x|}{R}\)^{\frac{N-p}{p-1}}\)^{\beta}} \,dx \)^{\frac{p}{p^*(s)}} \le \int_{B_R} | L_p u |^p \,dx
\end{align*}
with the differential operator $L_p$ is invariant under the scaling (\ref{new scale}). 
The scaling in \cite{I} looks different from the scaling (\ref{new scale}). However, taking $\la \mapsto \la^{-\frac{p-1}{N-p}}$, we observe that these scalings are same essentially. 
Note that $\| \nabla u \|_{L^p(B_R)}$ is not invariant under the scaling (\ref{new scale}) for non-radial functions. 
Therefore, in \S \ref{Proof}, we investigate our minimization problem $I_a$ without the transformation (\ref{trans}).
\end{remark}

\begin{remark}
Actually, the original transformation is given in Theorem 18. in \cite{F} for $u \in W_{0, {\rm rad}}^{1,2}(B_1)$, where $B_1, \Omega \subset \re^2$ as follows.
\begin{align}\label{Ftrans}
w(y)= u(x),\,\, {\rm where}\,\, G_{\Omega, z} (y) = G_{B_1, 0}(x) = -\frac{1}{2\pi} \log |x|
\end{align}
and $G_{\Omega, z}(y)$ is the Green function in a domain $\Omega$, which has a singularity at $z \in \Omega$. 
Therefore, we can observe that the transformation (\ref{trans}) is (\ref{Ftrans}) in the case where $W_{0, {\rm rad}}^{1,p}(B_1), p < N$ and $z=0, B_1 \subset \re^N = \Omega$.  
\end{remark}


In addition to (\ref{Ftrans}) and (\ref{trans}), 
there are various transformations by which the classical Hardy-Sobolev type inequality equivalently connects to another inequality for radial functions (ref. \cite{Z}, \cite{HK}, \cite{II}, \cite{ST}, \cite{I}, \cite{S}). 
Next, we explain them comprehensively. 
An unified viewpoint is {\it to connect two typical functions on each world (e.g. fundamental solution of $p-$Laplacian, virtual minimizer of the Hardy type inequality)}. 
In \cite{Z}, \cite{HK} and \cite{II}, the following transformation (\ref{HK trans}) is considered by using two fundamental solutions of $p-$Laplacian and weighted $p-$Laplacian: div$( |x|^{p-N} |\nabla u |^{p-2} \nabla u)$.
\begin{align}\label{HK trans}
u(r) = w(t), \,\,\text{where}\,\, t^{-\frac{N-p}{p-1} } = \log \frac{R}{r}.
\end{align}
Then they obtain the equality of two norms between the subcritical Sobolev space $W_0^{1, p} (\re^N) \,(p < N)$ and the weighted critical Sobolev space $W_0^{1, p} (B_R^N ; |x|^{p-N} \,dx)$ as follows.
\begin{align*}
\int_{\re^N} | \nabla w |^p \,dy = \int_{B^N_R} |x|^{p-N} | \nabla u |^p \,dx, \,\,
\int_{\re^N} \frac{| w |^{q}}{|y|^s} \,dy 
= \frac{p-1}{N-p} \int_{B_R^N} \frac{| u |^{q}}{|x|^N \( \log \frac{R}{|x|}  \)^{\beta (s)}} \,dx.
\end{align*}
On the other hand, in \cite{ST} and \cite{S}, they consider the following transformation (\ref{ST trans}) by using two fundamental solutions of $p \,(=N) -$Laplacian and $N-$Laplacian as follows.
\begin{align}\label{ST trans}
u(r) = w(t), \,\,\text{where}\,\, t^{-\frac{m-N}{N-1} } = \log \frac{R}{r} \,\, \text{is equivalent to} \, t^{-\frac{m-N}{N} } = \( \log \frac{R}{r} \)^{\frac{N-1}{N}} .
\end{align}
A different point of the transformation (\ref{ST trans}) from these transformations (\ref{Ftrans}), (\ref{trans}), (\ref{HK trans}) is to consider {\it the difference of dimensions on each world}. Thanks to the difference of dimensions, we obtain the equality of two norms between the critical Sobolev space $W_0^{1, N} (B_R^N)$ and the higher dimensional subcritical Sobolev space $W_0^{1, N} (\re^m) \,(N < m)$ as follows.
\begin{align*}
\int_{\re^m} | \nabla w |^N \,dy = \int_{B^N_R} | \nabla u |^N \,dx,\,\,
\int_{\re^m} \frac{| w |^{q}}{|y|^{\alpha}} \,dy 
=\frac{\w_{m-1} (N-1)}{\w_{N-1} (m-N)} \int_{B_R^N} \frac{| u |^{q}}{|x|^N \( \log \frac{R}{|x|}  \)^{\gamma(\alpha )}} \,dx,
\end{align*}
where $\gamma(\alpha) = \frac{(m-1)N - (N-1) \alpha}{m-N}$. 
This gives an equivalence between the critical Hardy inequality and a part of the higher dimensional subcritical Hardy inequality (ref. \cite{ST}). 
Besides, this also gives an relationship between 
the embedding of the subcritical Sobolev space into the Lorentz spaces for $q>p$:
\begin{align*}
W_0^{1,p} (B^N_R) \hookrightarrow L^{p^*, p} \hookrightarrow L^{p^*, q} \hookrightarrow L^{p^*, \infty} 
\end{align*}
and the embedding of the critical Sobolev space into the Lorentz-Zygmund spaces for $q>N$:
\begin{align*}
W_0^{1,N} (B^N_R) \hookrightarrow L^{\infty, N}(\log L)^{-1} \hookrightarrow L^{\infty, q}(\log L)^{-1+ \frac{1}{N} - \frac{1}{q}} \hookrightarrow L^{\infty, \infty}(\log L)^{-1+\frac{1}{N}} = {\rm Exp L}^{\frac{N}{N-1}}.
\end{align*}
Since almost all transformations are applicable only for radial functions, 
we can expect the different phenomena from classical results for any functions, for example the existence and the non-existence of a minimizer. 
In fact, the author in \cite{S} shows the existence of a non-radial minimizer of the  inequality associated with the embedding: $W_0^{1,N} (B^N_R) \hookrightarrow L^{\infty, q}(\log L)^{-1+ \frac{1}{N} - \frac{1}{q}}$. 
In this paper, we study an analogue of this work \cite{S}.

Finally, as an application of this unified viewpoint for the transformations, we derive some {\it infinite dimensional form} of the classical Sobolev inequality in a different way from \cite{BP} which is a study of the logarithmic Sobolev inequality. 
In order to consider a limit as the dimension $m \to \infty$, we reduce the dimension $m$ to $N$ by using the following transformation (\ref{trans dim}) which connects two norms between the Sobolev space $W_0^{1,p}(\re^m)$ and the lower dimensional Sobolev space $W_0^{1,p}(\re^N)$, where $p < N < m$.
\begin{align}\label{trans dim}
u(r) = w(t), \,\,\text{where}\,\, t^{-\frac{N-p}{p-1} } = r^{-\frac{m-p}{p-1} }.
\end{align}
Then we can see that
\begin{align*}
\int_{\re^m} | \nabla u |^p \,dx 
&= \frac{\w_{m-1}}{\w_{N-1}} \( \frac{m-p}{N-p} \)^{p-1} \int_{\re^N} | \nabla w |^p \,dy,\\
\int_{\re^m} | u |^{\frac{mp}{m-p}} \,dx
&= \frac{\w_{m-1}}{\w_{N-1}} \,\frac{N-p}{m-p} \int_{\re^N} \frac{| w |^{\frac{mp}{m-p}}}{|y|^{\frac{m-N}{m-p}p}} \,dy.
\end{align*}
Therefore the Sobolev inequality (\ref{HS}) for radial functions $u \in W_0^{1,p}(\re^m)$:
\begin{equation*}
C_{m, p, 0} \( \int_{\re^m} |u|^{\frac{mp}{m-p}} dx \)^{\frac{m-p}{m}} \le \int_{\re^m} | \nabla u |^p dx
\end{equation*}
is equivalent to the following inequality for radial functions $w \in W_0^{1,p}(\re^N)$.
\begin{align}\label{S another}
C_{m, p, 0} \( \frac{\w_{N-1}}{\w_{m-1}} \)^{\frac{p}{m}} \( \frac{N-p}{m-p} \)^{p-\frac{p}{m}} \( \int_{\re^N} \frac{| w |^{\frac{mp}{m-p}}}{|x|^{\frac{m-N}{m-p}p}} \,dy \)^{\frac{m-p}{m}} \le \int_{\re^N} | \nabla w |^p \, dy.
\end{align}
Since
\begin{align*}
&C_{m,p,0}=\pi^{\frac{p}{2}} m \( \frac{m-p}{p-1} \)^{p-1} \( \frac{\Gamma (\frac{m}{p}) \Gamma (m+ 1 -\frac{m}{p})}{\Gamma (m)  \Gamma(1+\frac{m}{2})} \)^{\frac{p}{m}}\,\,(\text{Sobolev's best constant}), \\
&\w_{N-1} = \frac{N \pi^{\frac{N}{2}}}{\Gamma \( 1+ \frac{N}{2} \)}, \,\, 
\Gamma (t) = \sqrt{2\pi} \,t^{\,t-\frac{1}{2}} \,e^{-t} + o(1) \,\, \text{as}\,\, t \to \infty \,\,(\text{Stirling's formula}),
\end{align*}
we have
\begin{align*}
&C_{m, p, 0} \( \frac{\w_{N-1}}{\w_{m-1}} \)^{\frac{p}{m}} \( \frac{N-p}{m-p} \)^{p-\frac{p}{m}} \\
&= \frac{m}{m-p} \,\frac{(N-p)^p}{(p-1)^{p-1}} \( \frac{\w_{N-1} \,(m-p)}{N-p} \)^{\frac{p}{m}}  \( \frac{\Gamma (\frac{m}{p}) \Gamma \( \frac{p-1}{p} m +1 \)}{\Gamma (m+1) } \)^{\frac{p}{m}} \\
&= \frac{(N-p)^p}{(p-1)^{p-1}} \( \frac{(\frac{m}{p})^{\frac{m}{p} -\frac{1}{2}} e^{-\frac{m}{p}} \( \frac{p-1}{p} m +1\)^{\frac{p-1}{p} m +\frac{1}{2}} e^{-\frac{p-1}{p} m -1}}{(m+1)^{m+\frac{1}{2}} e^{-(m+1)} } \)^{\frac{p}{m}} +o(1)\\
&= \( \frac{N-p}{p} \)^p+o(1) \,\,(m \to \infty).
\end{align*}
Hence we can obtain the limit of the left-hand side of (\ref{S another}) as $m \to \infty$ as follows.
\begin{align*}
C_{m, p, 0} \( \frac{\w_{N-1}}{\w_{m-1}} \)^{\frac{p}{m}} \( \frac{N-p}{m-p} \)^{p-\frac{p}{m}} \( \int_{\re^N} \frac{| w |^{\frac{mp}{m-p}}}{|y|^{\frac{m-N}{m-p}p}} \,dy \)^{\frac{m-p}{m}} \to
\( \dfrac{N-p}{p} \)^p \int_{\re^N} \dfrac{|w|^p}{|y|^p} dy.
\end{align*}

From above calculations, we can observe an interesting new aspect of the classical Hardy inequality, that is {\it an infinite dimensional form} of the classical Sobolev inequality. 
And we also see that under the transformation (\ref{trans dim}), the Hardy inequality on $W_0^{1,p}(\re^m)$ is equivalent to it on $W_0^{1,p}(\re^N)$, that is, the Hardy inequality is independent of the dimension in this sense.

%
%

\section{Proof of Theorem \ref{Main}: the existence and the non-existence of the minimizer}\label{Proof}

In this section, we prepare several Lemmas and Propositions and show Theorem \ref{Main}. 

Note that we can apply the rearrangement technique to our minimization problem $I_a$ for $a \in [0, \frac{s(p-1)}{p(N-1)}]$. More precisely, since the potential function $V_a(x)$ is radially decreasing on $B_R$ for $a \in [0, \frac{s(p-1)}{p(N-1)}]$, the P\'olya-Szeg\"o inequality and the Hardy-Littlewood inequality imply that
\begin{align*}
\dfrac{\int_{B_R} | \nabla u |^p \,dx}{\( \int_{B_R} |u|^{p^*(s)} V_a (x)\, dx \)^{\frac{p}{p^*(s)}}} \ge \dfrac{\int_{B_R} | \nabla u^{\#} |^p \,dx}{\( \int_{B_R} |u^{\#}|^{p^*(s)} V_a(x)\, dx \)^{\frac{p}{p^*(s)}}} \ge I_{a, {\rm rad}}
\end{align*}
for any $u \in W_0^{1,p}(B_R)$ and $a \in [0, \frac{s(p-1)}{p(N-1)}]$. Therefore we have
\begin{align}\label{rearrange}
I_a = I_{a, {\rm rad}} =C_{N, p, s} \quad \text{for any} \,\,a \in \left[0, \frac{s(p-1)}{p(N-1)} \right].
\end{align}
However, we can not apply the rearrangement technique to $I_a$ for $a \in ( \frac{s(p-1)}{p(N-1)}, 1]$. Therefore it seems difficult to reduce the radial setting in general.  

First, instead of rearrangement, we use the following lemma by which we can reduce the radial setting when $s=p$.

\begin{lemma}\label{radial nomi}
Let $1 < q <\infty$, $f=f(x)$ be a radial function on $B_R$. 
If there exists $C>0$ such that for any radial functions $u \in C_c^1(B_R)$ the inequality:
\begin{equation}\label{rad}
C \int_{B_R} |u|^q f(x)\, dx \le \int_{B_R} | \nabla u |^q\, dx 
\end{equation}
holds, then for any functions $w \in C_c^1(B_R)$ the inequality :
\begin{equation}\label{non-rad}
C \int_{B_R} |w|^q f(x)\, dx \le \int_{B_R} \left| \nabla w \cdot \frac{x}{|x|} \right|^q\, dx 
\end{equation}
holds.
\end{lemma}

\begin{proof}
For any $w \in C_c^1(B_R)$, define a radial function $W$ as follows.
\begin{align*}
W(r) = \( \w_{N-1}^{-1} \int_{\mathbb{S}^{N-1}} | w(r\w ) |^q\,dS_{\w} \)^{\frac{1}{q}} \quad (0 \le r \le R).
\end{align*}
Then we have
\begin{align*}
|W \,'(r) | &= \w_{N-1}^{-\frac{1}{q}} \( \int_{\mathbb{S}^{N-1}} | w(r\w ) |^q\,dS_{\w} \)^{\frac{1}{q}-1} \int_{\mathbb{S}^{N-1}} | w |^{q-1} \left| \frac{\pd w}{\pd r} \right| \,dS_{\w} \\
&\le \w_{N-1}^{-\frac{1}{q}} \( \int_{\mathbb{S}^{N-1}} \left| \frac{\pd w}{\pd r} (r\w ) \right|^q \,dS_{\w} \)^{\frac{1}{q}}.
\end{align*}
Therefore we have
\begin{align}\label{right}
\int_{B_R} | \nabla W |^q\, dx &\le \int_{B_R} \left| \nabla w \cdot \frac{x}{|x|} \right|^q\, dx, \\
\label{left}
\int_{B_R} |W|^q f(x)\, dx &= \int_{B_R} |w|^q f(x)\, dx.
\end{align}
From (\ref{rad}) for $W$, (\ref{right}), and (\ref{left}), we obtain (\ref{non-rad}) for any $w$.
\end{proof}


Second, we give a necessary and sufficient condition of the positivity of $I_a$ for $a \in [0, 1]$. As we see Proposition \ref{I_a rad}, $I_{a, {\rm rad}} = C_{N, p, s} >0$ for any $s \in [0,p]$ and any $a \in [0,1]$. However, $I_a$ is not so due to the boundary singularity. This is also mentioned by \cite{I}. For readers convenience, we give a sketch of the proof. 

\begin{prop}\label{Posi I_a} \,\,
$I_a = 0 \iff a =1$ and $0\le s < p$.
\end{prop}

\begin{proof}
Let $a=1$. In the similar way to \cite{S}, set $x_{\ep} = (R - 2\ep ) \frac{y}{R}$ for $y \in \pd B_R$ and for small $\ep >0$. 
Then we define $u_\ep$ as follows:
\begin{align*}
u_{\ep}(x) =
	   \begin{cases}
		v\( \frac{|x-x_{\ep}|}{\ep} \) \,\,\,&\text{if} \,\,\, x \in B_{\ep}(x_{\ep}), \\
		0  &\text{if} \,\,\, x \in B_R \setminus B_{\ep}(x_{\ep}),
	   \end{cases}
\,\, \text{where}\,\, v(t)=
	   \begin{cases}
		1 \,\,\,&\text{if} \,\,\,0\le t \le \frac{1}{2},  \\
		2(1-t)  &\text{if} \,\,\, \frac{1}{2} < t \le 1.
	   \end{cases}
\end{align*}
Then we have
\begin{align*}
\int_{B_R} | \nabla u_{\ep} (x)|^p \,dx &= \ep^{N-p} \int_{B_1} |\nabla v (|z|)|^p \,dz = C \ep^{N-p}, \\
\int_{B_R} \frac{|u_{\ep}(x)|^{p^*(s)}}{|x|^s \( 1- \,\( \frac{|x|}{R} \)^{\frac{N-p}{p-1}} \)^\beta} \,dx 
&\ge C \int_{B_\ep (x_\ep )} \frac{|u_{\ep}(x)|^{p^*(s)}}{(R-|x|)^{\beta}} dx \ge  \frac{C}{(3\ep )^{\beta}} \int_{B_{\frac{\ep}{2}} (x_\ep )} dx =C \, \ep^{N-\beta}.
\end{align*}
Hence we see that 
\begin{align*}
I_1 \le C \ep^{N-p- (N-\beta) \frac{p}{p^*(s)}} = C \ep^{\frac{N-1}{N-s} (p-s)} \to 0\,\,\text{as}\,\, \ep \to 0\,\,\text{if}\,\,  0\le s < p.
\end{align*}
Therefore $I_1 = 0$ if $a=1$ and $0\le s < p$. 
Conversely, we can easily show that $I_a > 0$ except for that case.
Indeed, if $s=p$, then $I_a = I_{a, {\rm rad}} > 0$ for any $a \in [0,1]$ from Proposition \ref{I_a rad} and Lemma \ref{radial nomi}. And also, if $a < 1$, then there is no boundary singularity. Thus $I_a > 0$ for any $a \in [0, 1)$. Therefore, we obtain the necessary and sufficient condition of the positivity of $I_a$. 
\end{proof}


Third, we show that $I_a$ is monotone decreasing and continuous with respect to $a \in [0,1]$. The potential function $V_a(x)$ is continuously monotone-increasing with respect to $a \in [0, 1)$. Thus it is easy to show the monotone-decreasing property of $I_a$ with respect to $a \in [0,1]$ and the continuity of $I_a$ with respect to $a \in [0, 1)$. Here, we give a proof of the continuity of $I_a$ at $a=1$ only. 

\begin{lemma}\label{conti I_a}
$I_a$ is monotone-decreasing and continuous with respect to $a \in [0, 1]$.
\end{lemma}

\begin{proof}
From the definition of $I_1$, we can take $(u_m)_{m=1}^\infty \subset C_c^{\infty}(B_R)$ and $R_m < R$ for any $m$ such that supp $u_m \subset B_{R_m}, R_m \nearrow R$, and
\begin{align*}
\dfrac{\int_{B_{R_m}} | \nabla u_m |^p \,dx}{\( \int_{B_{R_m}} \frac{|u_{m}(x)|^{p^*(s)}}{|x|^s \( 1- \,\( \frac{|x|}{R} \)^{\frac{N-p}{p-1}} \)^\beta}\, dx \)^{\frac{p}{p^*(s)}}} = I_1 + o(1) \quad {\rm as }\,\, m \to \infty.
\end{align*}
Set $\la_m = \frac{R}{R_m} >1$ and $v(y)= \la_m^{-\frac{N-p}{p}} u_m(x)$, where $y=\la_m x$. Then 
\begin{align*}
\dfrac{\int_{B_{R_m}} | \nabla u_m |^p \,dx}{\( \int_{B_{R_m}} \frac{|u_{m}(x)|^{p^*(s)}}{|x|^s \,\( 1- \,\( \frac{|x|}{R} \)^{\frac{N-p}{p-1}} \)^\beta}\, dx \)^{\frac{p}{p^*(s)}}}
= \dfrac{\int_{B_R} | \nabla v |^p \,dx}{\( \int_{B_R} \frac{|v|^{p^*(s)}}{|y|^s \,\( 1- a_m \( \frac{|y|}{R} \)^{\frac{N-p}{p-1}} \)^\beta}\, dy \)^{\frac{p}{p^*(s)}}} \ge I_{a_m},
\end{align*}
where $a_m = \la_m^{-\frac{N-p}{p-1}} \nearrow 1$ as $m \to \infty$. Therefore we have $I_{a_m} \le I_1 + o(1)$. Obviously, we have $I_1 \le I_{a_m}$ from the monotone-decreasing property of $I_a$. Hence we see that $\lim_{a \nearrow 1} I_a =I_1$. 
\end{proof}


Fourth, we show the following result for the Sobolev case where $s=0$, in the similar way to \cite{SSW} for the H\'enon problem.

\begin{prop}\label{Sobo} \,\,
Let $R < \infty$ and $f: B_R \to \re$ be a nonnegative bounded continuous function with $f \nequiv 0$. Then
\begin{align*}
S:= \inf_{u \in W_0^{1,p}(B_R ) \setminus \{ 0\} } \dfrac{\int_{B_R} |\nabla u |^{p} \,dx}{\( \int_{B_R} |u|^{\frac{Np}{N-p}} f(x) \,dx \)^{\frac{N-p}{N}}} = \( \max_{x \in \overline{B_R}} f(x) \)^{-\frac{N-p}{N}} C_{N, p, 0}
\end{align*}
and there is no minimizers of the minimization problem $S$.
\end{prop}

\begin{proof}
Since we easily obtain $S \ge \( \max_{x \in \overline{B_R}} f(x) \)^{-\frac{N-p}{N}} C_{N, p, 0}$, we shall show $S \le \( \max_{x \in \overline{B_R}} f(x) \)^{-\frac{N-p}{N}} C_{N, p, 0}$. 
Let $z \in \overline{B_R}$ be a maximum point of $f$. For simplicity, we assume that $z \in B_R$. 
For any $\ep > 0$ there exist $T>0$ and $v \in C_c^{\infty}(B_T)$ such that
\begin{align*}
C_{N, p, 0} 
\ge \dfrac{\int_{B_T} |\nabla v |^{p} \,dx}{\( \int_{B_T} |v|^{\frac{Np}{N-p}} \,dx \)^{\frac{N-p}{N}}} - \frac{\ep}{2}. 
\end{align*}
Set $u_\la (x) = \la^{\frac{N-p}{p}} v (\la (x-z))$ for $\la >0$. Then for large $\la >0$ we have
\begin{align*}
C_{N, p, 0} 
\ge \dfrac{\int_{B_{\la^{-1}T} (z)} |\nabla u_\la |^{p} \,dx}{\( \int_{B_{\la^{-1}T} (z)} |u_\la|^{\frac{Np}{N-p}} \,dx \)^{\frac{N-p}{N}}} - \frac{\ep}{2}
&\ge f(z)^{\frac{N-p}{N}} \dfrac{\int_{B_{\la^{-1}T} (z)} |\nabla u_\la |^{p} \,dx}{\( \int_{B_{\la^{-1}T} (z)} |u_\la|^{\frac{Np}{N-p}} f(x) \,dx \)^{\frac{N-p}{N}}} - \ep \\
&\ge f(z)^{\frac{N-p}{N}} S - \ep.
\end{align*}
Since $\ep$ is arbitrary, we obtain $S \le \( \max_{x \in \overline{B_R}} f(x) \)^{-\frac{N-p}{N}} C_{N, p, 0}$. The case where $z \in \pd B_R$ is also showed in the same way. We omit the proof in that case.


On the other hand, the non-attainability of $S$ comes from it of $C_{N, p, 0}$. 
Indeed, if we assume that $v \ge 0$ is a minimizer of $S$, then
\begin{align*}
S= \dfrac{\int_{B_R} |\nabla v |^{p} \,dx}{\( \int_{B_R} |v|^{\frac{Np}{N-p}} f(x) \,dx \)^{\frac{N-p}{N}}} 
&\ge \( \max_{x \in \overline{B_R}} f(x) \)^{-\frac{N-p}{N}} \dfrac{\int_{B_R} |\nabla u |^{p} \,dx}{\( \int_{B_R} |u|^{\frac{Np}{N-p}} \,dx \)^{\frac{N-p}{N}}} \\
&> \( \max_{x \in \overline{B_R}} f(x) \)^{-\frac{N-p}{N}} C_{N, p, 0} = S,
\end{align*}
where the last inequality comes from the non-attainability of $C_{N, p, 0}$. 
This is a contradiction.
\end{proof}


Fifth, we show the concentration level of minimizing sequences of $I_a$ is $I_{a, \text{rad}}$ when $0 < s < p$.

\begin{lemma}\label{concentration level}
Let $0 < s < p$ and $0 \le a<1$. 
If $I_a < I_{a, {\rm rad}} = C_{N, p, s}$, then $I_a$ is attained by a non-radial function.
\end{lemma}

In order to show Lemma \ref{concentration level} also in the case where $p \neq 2$, we prepare two Lemmas. Lemma \ref{BM} is concerning with almost everywhere convergence of the gradients of a sequence of solutions. This guarantees to use Lemma \ref{BL} in the proof of Lemma \ref{concentration level}. 

%
%
\begin{lemma}(\cite{Brezis-Lieb}) \label{BL}
For $p \in (0, +\infty)$,
let $(g_m)_{m=1}^{\infty} \subset L^p(\Omega, \mu)$ be a sequence of functions
on a measurable space $(\Omega, \mu)$ such that
\begin{enumerate}
\item[(i)] $\| g_m \|_{L^p(\Omega, \mu)} \le \ ^{\exists} C < \infty$ for all $m \in \N$, and
\item[(ii)] $g_m(x) \to g(x)$ $\mu$-a.e. $x \in \Omega$ as $m \to \infty$.
\end{enumerate}
Then
\begin{equation*}
	\lim_{m \to \infty} \( \| g_m \|_{L^p(\Omega, \mu)}^p - \| g_m -g \|_{L^p(\Omega, \mu)}^p \) = \| g \|_{L^p(\Omega, \mu)}^p.
\end{equation*}
\end{lemma}
Note that we can apply Lemma \ref{BL} to $\mu(dx) = f(x) dx$, where $f$ is any nonnegative $L^1(\Omega)$ function.

%
%
\begin{lemma}(\cite{Boccardo-Murat} Theorem 2.1.)\label{BM}
Let $(u_m)_{m=1}^{\infty} \subset W^{1,p}_0(\Omega)$ be such that, 
as $m \to \infty$,
$u_m \rightharpoonup u$ weakly in $W^{1,p}_0(\Omega)$ and satisfies
\begin{equation*}
	- \lap_p u_m  = g_m + f_m \quad \text{in} \,\,\,\D^{\prime}(\Omega),
\end{equation*}
where $f_m \to 0$ in $W^{-1,p^{\prime}}_0(\Omega)$ and $g_m$ is bounded in $\M (\Omega)$, the space of Radon measures on $\Omega$, i.e.
\begin{equation*}
|< g_m, \phi > | \leq C_K \| \phi \|_{\infty}
\end{equation*}
for all $\phi \in \D (\Omega)$ with $\text{supp} \; \phi \subset K$. 
Then there exists a subsequence, say $u_{m_k}$, such that
\begin{equation*}
	u_{m_k} \to u \quad \text{in} \,\,\,W^{1,\gamma}_0(\Omega ) \quad ({}^{\forall}\gamma < p).
\end{equation*}
\end{lemma}

Before showing Lemma \ref{concentration level}, we apply Lemma \ref{BM} for a minimizing sequence of $I_a$. 
Set
\begin{align*}
J(u) = \| \nabla u \|_{L^p(B_R)}^{p^*(s)} - I_a^{\frac{p^*(s)}{p}} \int_{B_R} |u|^{p^*(s)} V_a(x) \,dx \quad \text{for}\,\,u \in W_0^{1,p}(B_R).
\end{align*}
From the definition of $I_a$, we see that ${\rm inf}_{u \in W_0^{1,p}(B_R)} J(u) = 0$. And
\begin{align*}
J'(u)[\varphi] = p^*(s)\, \| \nabla u \|_{L^p(B_R)}^{p^*(s)-p} \int_{B_R} |\nabla u|^{p-2} \nabla u \cdot \nabla \varphi \,dx - I_a^{\frac{p^*(s)}{p}} p^*(s) \int_{B_R} |u|^{p^*(s)-2} u \varphi V_a(x) \,dx
\end{align*}
for $\varphi \in \( W_0^{1,p} \)^*$. Thus we observe that $J \in C^1 \( W_0^{1,p} \,; \re \)$. Let $( u_m )_{m=1}^\infty \subset W_0^{1,p}(B_R)$ be a minimizing sequence of $I_a$ with $\int_{B_R} |u_m|^{p^*(s)} V_a(x) \,dx =1$ for any $m \in \N$ and $\| \nabla u_m \|^p_{L^p(B_R)} = I_a + o(1)$ as $m \to \infty$. 
By Ekeland's Variational Principle (see e.g. \cite{Struwe}), there exists $( w_m )_{m=1}^\infty \subset W_0^{1,p}(B_R)$ such that 
\begin{align*}
&(i) \,\,0 \le J(w_m) \le J(u_m) = o(1) \quad (\,m \to \infty \,), \\
&(ii) \,\,\| J'(w_m) \|_{(W_0^{1,p})^*} =o (1) \quad ( \,m \to \infty\,),\\
&(iii)\,\, \| \nabla (w_m -u_m) \|_{L^p(B_R)} = o(1)\quad ( \,m \to \infty\,).
\end{align*}
From (iii), we see that $( w_m )_{m=1}^\infty$ is a minimizing sequence of $I_a$ with $\int_{B_R} |w_m|^{p^*(s)} V_a(x) \,dx =1 + o(1)$ and $\| \nabla w_m \|^p_{L^p(B_R)} = I_a + o(1)$ as $m \to \infty$. 
Let $w_m  \rightharpoonup w$ in $W_0^{1,p}(B_R)$ as $m \to \infty$, passing to a subsequence if necessary. From (ii), for any $\varphi \in W_0^{1,p}(B_R)$ we have 
\begin{align*}
\int_{B_R} |\nabla w_m|^{p-2} \nabla w_m \cdot \nabla \varphi \,dx - I_a^{\frac{p^*(s)}{p}} \| \nabla w_m \|_{L^p(B_R)}^{p-p^*(s)} \int_{B_R} |w_m|^{p^*(s)-2} w_m \varphi V_a(x) \,dx = o(1)\,
\end{align*}
which yields that $w_m$ satisfies
\begin{align*}
-{\rm div} (\,|\nabla w_m|^{p-2} \nabla w_m ) = I_a^{\frac{p^*(s)}{p}} \| \nabla w_m \|_{L^p(B_R)}^{p-p^*(s)} |w_m |^{p-2} w_m V_a(x) + f_m \quad \text{in}\,\, \mathcal{D}'(B_R)
\end{align*}
and $f_m \to 0$ in $W^{-1,p^{\prime}}_0(B_R)$. 
From Lemma \ref{BM}, passing to a subsequence if necessary, we have $\nabla w_m \to \nabla w$ a.e. in $B_R$. As a consequence, we can apply Lemma \ref{BL} for $\nabla w_m$ in the proof of Lemma \ref{concentration level}.


\begin{proof}[Proof of Lemma \ref{concentration level}]
Take a minimizing sequence $(u_m)_{m=1}^{\infty} \subset W_0^{1,p}(B_R)$ of $I_a$. Without loss of generality, we can assume that
\begin{align*}
\int_{B_R} |u_m|^{p^*(s)} V_a(x) \,dx =1,\quad \int_{B_R} |\nabla u_m|^p \,dx = I_a + o(1) \,\,{\rm as}\,\, m \to \infty.
\end{align*}
Since $(u_m)$ is bounded in $W_0^{1,p}(B_R)$, passing to a subsequence if necessary, $u_m \rightharpoonup u$ in $W_0^{1,p}(B_R)$ as $m \to \infty$. 
Replacing $u_m$ with $w_m$ (we write $u_m$ again) and applying Lemma \ref{BL}, we have
\begin{align*}
I_a &= \int_{B_R} |\nabla u_m|^p \,dx + o(1)\\
&= \int_{B_R} |\nabla (u_m -u)|^p \,dx + \int_{B_R} |\nabla u|^p \,dx +o(1) \\
&\ge I_a \( \int_{B_R} |u_m-u|^{p^*(s)} V_a (x)\, dx \)^{\frac{p}{p^*(s)}} + I_a \( \int_{B_R} |u|^{p^*(s)} V_a (x)\, dx \)^{\frac{p}{p^*(s)}} +o(1) \\
&\ge I_a \( \int_{B_R} \( |u_m-u|^{p^*(s)} + |u|^{p^*(s)} \) V_a(x) \,dx \)^{\frac{p}{p^*(s)}} +o(1) \\
&= I_a \( \int_{B_R} |u_m|^{p^*(s)} V_a (x)\, dx \)^{\frac{p}{p^*(s)}} +o(1) =I_a
\end{align*}
which implies that either $u \equiv 0$ or $u_m \to u \nequiv 0$ in $L^{p^*(s)}(B_R; V_a(x) dx)$ holds true from the equality condition of the last inequality and the positivity of $I_a$ for $a \in [0,1)$. 
We shall show that $u \nequiv 0$. Assume that $u \equiv 0$. 
Then we claim that
\begin{align}\label{concentrate}
I_{a, {\rm rad}} \le \int_{B_R} |\nabla u_m|^p \,dx +o(1).
\end{align}
If the claim (\ref{concentrate}) is true, then we see that $I_{a, {\rm rad}} \le I_a$ which contradicts the assumption. Therefore $u \nequiv 0$ which implies that $u_m \to u \nequiv 0$ in $L^{p^*(s)}(B_R; V_a(x) dx)$. 
Hence we have 
\begin{align*}
1= \int_{B_R} |u|^{p^*(s)} V_a(x)\, dx, \quad \int_{B_R} |\nabla u|^p \,dx \le \liminf_{m \to \infty} \int_{B_R} |\nabla u_m|^p \,dx = I_a.
\end{align*}
Thus we can show that $u$ is a minimizer of $I_a$. 
We shall show the claim (\ref{concentrate}). 
Since $u_m \to 0$ in $L^r (B_R)$ for any $r \in [1, p^* (0))$ and the potential function $V(x)$ is bounded away from the origin, for any small $\ep > 0$ we have
\begin{align*}
1= \int_{B_R} |u_m|^{p^*(s)} V_a (x)\, dx = \int_{B_{\frac{\ep R}{2}}} |u_m|^{p^*(s)} V_a (x)\, dx +o(1).
\end{align*}
Let $\phi_\ep$ be a smooth cut-off function which satisfies the followings:
\begin{align*}
0 \le \phi_\ep \le 1, \,\,\phi_\ep \equiv 1 \,\,{\rm on}\,\,B_{\frac{\ep R}{2}}(0),\,\,{\rm supp}\, \phi_\ep \subset B_{\ep R} (0), \,\,|\nabla \phi_\ep | \le C\ep^{-1}. 
\end{align*}
Set $\tilde{u_m}(y)=u_m(x)$ and $\tilde{\phi_\ep}(y)=\phi_\ep (x)$, where $y=\frac{x}{\ep}$. Then we have
\begin{align*}
1&= \( \int_{B_{\frac{\ep R}{2}}} |u_m|^{p^*(s)} V_a (x)\, dx \)^{\frac{p}{p^*(s)}} +o(1) \\
&\le \( \int_{B_{\ep R} } \frac{|u_m \phi_\ep |^{p^*(s)}}{|x|^s\( 1-a\, \(\frac{|x|}{R}\)^{\frac{N-p}{p-1}}\)^{\beta}}\, dx \)^{\frac{p}{p^*(s)}} +o(1) \\
&=\( \int_{B_R } \frac{|\tilde{u_m} \tilde{\phi_\ep} |^{p^*(s)}}{|x|^s\( 1-a\ep^{\frac{N-p}{p-1}} \, \(\frac{|x|}{R}\)^{\frac{N-p}{p-1}}\)^{\beta}}\, dx \)^{\frac{p}{p^*(s)}} +o(1) 
\le I_{a\ep^{\frac{N-p}{p-1}}}^{-1} \int_{B_R} |\nabla (\tilde{u_m} \tilde{\phi_\ep } ) |^p \,dx +o(1).
\end{align*}
We see that $a \ep^{\frac{N-p}{p-1}} \le \frac{s(p-1)}{p(N-1)}$ for small $\ep$. Since $I_{a\ep^{\frac{N-p}{p-1}}} =I_{a,\,{\rm rad}}$ for small $\ep$ by (\ref{rearrange}), we have
\begin{align*}
1&\le I_{a, {\rm rad}}^{-1} \int_{B_R} |\nabla (\tilde{u_m} \tilde{\phi_\ep } ) |^p \,dx +o(1) \\
&\le I_{a, {\rm rad}}^{-1} \( \int_{B_{\ep R}} |\nabla u_m |^p \,dx + C \int_{B_{\ep R}} |\nabla u_m |^{p-1} | \nabla \phi_\ep | |u_m| \phi_\ep^{p-1} + |u_m |^p | \nabla \phi_\ep |^p \,dx \) +o(1) \\
&\le I_{a, {\rm rad}}^{-1} \( \int_{B_{\ep R}} |\nabla u_m |^p \,dx + pC\ep^{-1} \| \nabla u_m \|_{L^p}^{p-1} \| u_m \|_{L^p} + C \ep^{-p} \| u_m \|_{L^p}^p \) + o(1)\\
&\le I_{a, {\rm rad}}^{-1} \int_{B_{\ep R}} |\nabla u_m |^p \,dx + o(1) 
\le I_{a, {\rm rad}}^{-1} \int_{B_R} |\nabla u_m |^p \,dx +o(1).
\end{align*}
Therefore we obtain the claim (\ref{concentrate}). The proof of Lemma \ref{concentration level} is now complete.
\end{proof}

Finally, we give a proof of Theorem \ref{Main}.

\begin{proof}[Proof of Theorem \ref{Main}]\textcolor{white}{a}\\
(i) Let $s=p$. From Lemma \ref{radial nomi} and Proposition \ref{I_a rad}, we easily obtain $I_a = I_{a, {\rm rad}} = C_{N, p, p} = (\frac{N-p}{p})^p$ and the non-attainability of $I_a$. We omit the proof.

\noindent
(ii) Let $s=0$. From Proposition \ref{Sobo}, we obtain $I_a = C_{N, p, 0} (1-a)^{\frac{N-1}{N} p}$ and the non-attainability of $I_a$.

\noindent
(iii) Let $0 < s < p$. Note that $I_1 = 0$ by Proposition \ref{Posi I_a} and $I_a = I_{a, {\rm rad}} = C_{N, p, s}$ at least for $a \in [0, \frac{s(p-1)}{p(N-1)}]$ by (\ref{rearrange}). Since $I_a$ is continuous and monotone decreasing with respect to $a \in [0, 1]$ by Lemma \ref{conti I_a}, there exists $a_* \in [\frac{s(p-1)}{p(N-1)}, 1)$ such that $I_a < I_{a, {\rm rad}} = C_{N, p, s}$ for $a \in (a_*, 1)$ and $I_a = I_{a, {\rm rad}} = C_{N, p, s}$ for $a \in [0, a_*]$. 
Hence $I_a$ is attained by a non-radial function for $a \in (a_*, 1)$ by Lemma \ref{concentration level}. 
On the other hand, if we assume that there exists a nonnegative minimizer $u$ of $I_a$ for $a < a_*$, then we can show that at least, $u \in C^1(B_R \setminus \{ 0\})$ and $u >0$ in $B_R \setminus \{ 0\}$ by standard regularity argument and strong maximum principle to the Euler-Lagrange equation (\ref{EL}), see e.g. \cite{D}, \cite{PS}. 
Therefore we see that
\begin{align*}
I_{a,{\rm rad}} = I_a = \dfrac{\int_{B_R} | \nabla u |^p \,dx}{\( \int_{B_R} \frac{| u |^{p^*(s)}}{|x|^s\( 1-a\, \(\frac{|x|}{R}\)^{\frac{N-p}{p-1}}\)^{\beta}}\, dx \)^{\frac{p}{p^*(s)}}} 
> \dfrac{\int_{B_R} | \nabla u |^p \,dx}{\( \int_{B_R} \frac{| u |^{p^*(s)}}{|x|^s\( 1-a_* \(\frac{|x|}{R}\)^{\frac{N-p}{p-1}}\)^{\beta}}\, dx \)^{\frac{p}{p^*(s)}}} 
\ge I_{a, {\rm rad}}.
\end{align*}
This is a contradiction. Therefore $I_a$ is not attained for $a \in [0, a_*)$. 

Finally, we show that $a_* > A > \frac{s(p-1)}{p(N-1)}$, where $A$ is defined in Theorem \ref{Main} (iii). 
Note that the potential function $V_a$ is increasing with respect to $a \in [0, 1]$. 
Since $V_a(x)$ has one critical point at $|x|=R_a := \( \frac{s (p-1)}{a p (N-1)} \)^{\frac{p-1}{N-p}} R$, $V_a$ is decreasing for $|x| < R_a$ and increasing for $|x| > R_a$. Therefore we see that $R^{-s} ( 1-a)^{-\beta} = V_a (R) = V_1 (R_1) = R^{-s} \( \frac{s (p-1)}{p (N-1)} \)^{-s} \( 1- \frac{s (p-1)}{p (N-1)} \)^{-\beta}$ for $a = A$ by the shape of $V_a$. Let $\tilde{R} \in (0, R_1)$ satisfy $V_A(\tilde{R}) = V_A (R) = V_1 (R_1)$. 
Then we have
$
V_A^{\#} (x) = V_A (x) \,\, \text{for}\,\, x \in B_{\tilde{R}}.
$
Since $V_A^{\#}(x)$ is decreasing for $x \in B_R \setminus B_{\tilde{R}}$, $V_1 (x)$ is increasing for $x \in B_R \setminus B_{R_1}$, and $\tilde{R} < R_1$, we see that 
\begin{align*}
V_a^{\#}(x) < V_1 (x) \,\, \text{for any}\,\, x \in B_R\,\,\text{at least for}\,\, a \in [A, A+\ep].
\end{align*}
Then for any $u \in W_0^{1,p}(B_R)$ we have
\begin{align*}
\dfrac{\int_{B_R} | \nabla u |^p \,dx}{\( \int_{B_R} | u |^{p^*(s)} V_{A+\ep} (x) \, dx \)^{\frac{p}{p^*(s)}}} 
&\ge \dfrac{\int_{B_R} | \nabla u^{\#} |^p \,dx}{\( \int_{B_R} | u^{\#} |^{p^*(s)} V_{A+\ep}^{\#} (x) \, dx \)^{\frac{p}{p^*(s)}}} \\
&\ge \dfrac{\int_{B_R} | \nabla u^{\#} |^p \,dx}{\( \int_{B_R} | u^{\#} |^{p^*(s)} V_1 (x) \, dx \)^{\frac{p}{p^*(s)}}} \ge I_{1, \text{rad}}.
\end{align*}
Hence we have $I_{A+\ep} \ge I_{1, \text{rad}} = C_{N,p,s}$ which implies that $a_* > A$.
\end{proof}


\begin{remark}
We can also show Theorem \ref{Main} for general bounded domains with Lipschitz boundary in the similar way in \cite{S}, since we can generalize Proposition \ref{Posi I_a} to such domains.
\end{remark}

%
%

\section{Appendix}\label{App}

In this section, we give remarks of the following two nonlinear scalings (\ref{scale N}), (\ref{scale p}) for non-radial functions $v,w$ on the unit ball $B_1$.
\begin{align}\label{scale N}
v_\la (x) &= \la^{-\frac{N-1}{N}} v(y), \,{\rm where} \,\,y= \( \frac{|x|}{b} \)^{\la -1} x,  \,\,b \ge 1,\,\, \la \le 1,\\
\label{scale p}
w_\la (x) &= \la^{\frac{N-p}{p}} w(y), \,{\rm where} \,\,y=  \la \left[ 1 + a \( \la^{\frac{N-p}{p-1}} -1 \)  |x|^{\frac{N-p}{p-1}}  \right]^{\,-\frac{p-1}{N-p}} x, \,\,0\le a \le 1,\,\, \la \ge 1.
\end{align}
Note that each $y$ in (\ref{scale N}), (\ref{scale p}) is in $B_1$ thanks to the restriction of the length of $\la$. If $b=1$ or $a=1$ in (\ref{scale N}), (\ref{scale p}), then we do not need to restrict the length of $\la$. 
From \S \ref{Trans}, we see that two scalings (\ref{scale N}), (\ref{scale p}) are equivalent to the usual scaling: 
\begin{align}\label{usual scale}
u_\la (x)=\la^{\frac{N-p}{p}} u(y), \,{\rm where} \,\,y= \la x,\,\, \la >0.
\end{align}
only for radial functions $u$ by the transformations (\ref{ST trans}), (\ref{trans}). 
However each response of each derivative norm $\| \nabla (\cdot)\|_{L^N}$, $\| \nabla (\cdot)\|_{L^p}$ to each scaling (\ref{scale N}), (\ref{scale p}) is different for non-radial functions. This is a different structure between the minimization problem $I_a$ in this paper and it in \cite{S}, which is related to the embedding: $W_{0}^{1,N}(B_1) \hookrightarrow L^{q} \( B_1; \frac{dx}{|x|^N (\log \frac{b}{|x|})^{\frac{N-1}{N}q +1}} \)$ with $q \ge N$. 
More precisely, we show the followings.

\begin{prop}\label{bdd}
Let $b \ge 1$ and $a \in [0, 1]$. 
Then $\{ v_{\frac{1}{m}} \}_{m=1}^\infty \subset W_{0}^{1,N}(B_1)$ is unbounded for $v \in C_c^{1}(B_1)$ with $\nabla_{\mathbb{S}^{N-1}}v \nequiv 0$. On the other hand, $\{ w_m \}_{m=1}^\infty \subset W_{0}^{1,p}(B_1)$ is bounded for any $w \in C_c^{1}(B_1)$. 
\end{prop}

\begin{proof}[Proof of Proposition \ref{bdd}]
We use the polar coordinates. 
Let $r=|x|, t= |y|, \w \in \mathbb{S}^{N-1}$. By the scaling (\ref{scale N}), we have
\begin{align*}
t= b^{1-\la} r^\la \iff r= b^{1-\frac{1}{\la}} t^{\frac{1}{\la}}.
\end{align*}
Then we have
\begin{align*}
\frac{dt}{dr} &= \la \,\frac{t}{r} \to 0 \quad (\la \to 0).
\end{align*}
Therefore we see that for $\la \le 1$
\begin{align*}
\int_{B_1} |\nabla v_\la (x)|^{N} \,dx
&= \int_{B_R} |\nabla v_\la (x)|^{N} \,dx \\
&=  \int_{\mathbb{S}^{N-1}} \int_0^R \left| \frac{\pd v_\la}{\pd r} \w + \frac{1}{r} \nabla_{\mathbb{S}^{N-1}} v_\la \,\right|^N r^{N-1} \,dr dS_{\w}\\
&= \la^{1-N} \int_{\mathbb{S}^{N-1}} \int_0^1 \left| \frac{\pd v}{\pd t} \w + \( r \frac{dt}{dr} \)^{-1} \nabla_{\mathbb{S}^{N-1}} v \,\right|^N \( r \frac{dt}{dr} \)^{N-1} \,dt dS_{\w} \\
&= \int_{\mathbb{S}^{N-1}} \int_0^1 \left| \frac{\pd v}{\pd t} \w + \( r \frac{dt}{dr} \)^{-1} \nabla_{\mathbb{S}^{N-1}} v \,\right|^N t^{N-1} \,dt dS_{\w} \\
&= \int_{\mathbb{S}^{N-1}} \int_0^1 \left| \frac{\pd v}{\pd t} \w + \frac{1}{\la t}  \nabla_{\mathbb{S}^{N-1}} v \,\right|^N t^{N-1} \,dt dS_{\w} \to \infty \quad (\la \to 0).
\end{align*}
Hence we observe that $\{ w_{\frac{1}{m}} \}_{m=1}^\infty \subset W_{0}^{1,N}(B_1)$ is unbounded. 


On the other hand, by the scaling (\ref{scale p}), we have
\begin{align*}
t= \la r \left[ 1 + a \( \la^{\frac{N-p}{p-1}} -1 \)  r^{\frac{N-p}{p-1}}  \right]^{\,-\frac{p-1}{N-p}} = \la \left[ r^{-\frac{N-p}{p-1}} + a \( \la^{\frac{N-p}{p-1}} -1 \)    \right]^{\,-\frac{p-1}{N-p}}
\end{align*}
which is equal to 
\begin{align*}
r^{\,-\frac{N-p}{p-1}} = \la^{\frac{N-p}{p-1}} t^{-\frac{N-p}{p-1}} - a (\la^{\frac{N-p}{p-1}} -1).
\end{align*}
Then we have
\begin{align*}
\frac{dt}{dr} &= \frac{t}{r} \,\left[ 1 + a \( \la^{\frac{N-p}{p-1}} -1 \)  r^{\frac{N-p}{p-1}}  \right]^{-1}\\
&= \frac{t}{r} \,\left[ 1 + \frac{a \( \la^{\frac{N-p}{p-1}} -1 \)}{ \la^{\frac{N-p}{p-1}} t^{-\frac{N-p}{p-1}} - a (\la^{\frac{N-p}{p-1}} -1) } \right]^{-1}\\
&= \frac{t}{r} \,\left[ 1 + \frac{a}{ \frac{\la^{\frac{N-p}{p-1}}}{\( \la^{\frac{N-p}{p-1}} -1 \)} t^{-\frac{N-p}{p-1}} - a } \right]^{-1} 
\to \frac{t}{r} \,\left[ 1 + \frac{a}{ t^{-\frac{N-p}{p-1}} - a } \right]^{-1} \quad (\la \to \infty).
\end{align*}
Therefore we see that for $\la \ge 1$
\begin{align*}
\int_{B_1} |\nabla w_\la (x)|^{p} \,dx
&= \int_{B_R} |\nabla w_\la (x)|^{p} \,dx \\
&=  \int_{\mathbb{S}^{N-1}} \int_0^R \left| \frac{\pd w_\la}{\pd r} \w + \frac{1}{r} \nabla_{\mathbb{S}^{N-1}} w_\la \,\right|^p r^{N-1} \,dr dS_{\w}\\
&= \la^{N-p} \int_{\mathbb{S}^{N-1}} \int_0^1 \left| \frac{\pd w}{\pd t} \w + \( r \frac{dt}{dr} \)^{-1} \nabla_{\mathbb{S}^{N-1}} w \,\right|^p \( \frac{dt}{dr} \)^{p-1} r^{N-1} \,dt dS_{\w} \\
&= \int_{\mathbb{S}^{N-1}} \int_0^1 \left| \frac{\pd w}{\pd t} \w + \( r \frac{dt}{dr} \)^{-1} \nabla_{\mathbb{S}^{N-1}} w \,\right|^p t^{N-1} \,dt dS_{\w} \\
&\to \int_{\mathbb{S}^{N-1}} \int_0^1 \left| \frac{\pd w}{\pd t} \w + \frac{1}{t} \,\left[ 1 + \frac{a}{ t^{-\frac{N-p}{p-1}} - a } \right] \nabla_{\mathbb{S}^{N-1}} w \,\right|^p t^{N-1} \,dt dS_{\w}\,\,(\la \to \infty).
\end{align*}
Hence we observe that $\{ w_m \}_{m=1}^\infty \subset W_{0}^{1,p}(B_1)$ is bounded. 
\end{proof}


As a consequence of Proposition \ref{bdd}, we can also construct a non-compact non-radial sequence of the embedding: $W_0^{1,p}(B_1) \hookrightarrow L^{p^*(s)}(B_1; V_a(x)\,dx)$ which is related to our minimization problem $I_a$. 
On the other hand, in the view of the scaling, it seems difficult to construct a non-compact non-radial sequence of the embedding: $W_{0}^{1,N}(B_1) \hookrightarrow L^{q} \( B_1; \frac{dx}{|x|^N (\log \frac{b}{|x|})^{\frac{N-1}{N}q +1}} \)$ which is related to the minimization problem in \cite{S}. Therefore the following natural question arises.\\

\noindent
{\it Let $N=2$ and $b>1$. Consider the orthogonal decomposition: $W_{0}^{1,2}(B_1) = W_{0, {\rm rad}}^{1,2}(B_1) \,\oplus \, \( W_{0, {\rm rad}}^{1,2}(B_1) \)^{\perp}$. Is the embedding: $\( W_{0, {\rm rad}}^{1,2}(B_1) \)^{\perp} \hookrightarrow L^{q} \( B_1; \frac{dx}{|x|^N (\log \frac{b}{|x|})^{\frac{N-1}{N}q +1}} \)$ compact?}\\

\noindent
If the above question is affirmative, we may say that ``Non-radial compactness'' occurs in some sense. That is an opposite phenomenon of Strauss's radial compactness (ref. \cite{St}). Thus it may be interesting if the question is affirmative.


\section*{Acknowledgment}
This work was (partly) supported by Osaka City University Advanced
Mathematical Institute (MEXT Joint Usage/Research Center on Mathematics
and Theoretical Physics). 
The author was supported by JSPS KAKENHI Early-Career Scientists, No. JP19K14568.


\end{document}